\documentclass[reqno, 12pt]{amsart}
\pdfoutput=1
\makeatletter
\let\origsection=\section \def\section{\@ifstar{\origsection*}{\mysection}} 
\def\mysection{\@startsection{section}{1}\z@{.7\linespacing\@plus\linespacing}{.5\linespacing}{\normalfont\scshape\centering\S}}
\makeatother        

\usepackage{amsmath,amssymb,amsthm}
\usepackage{mathrsfs}
\usepackage{mathabx}\changenotsign
\usepackage{dsfont}
 
\usepackage{xcolor}
\usepackage[backref]{hyperref}
\hypersetup{
    colorlinks,
    linkcolor={red!60!black},
    citecolor={green!60!black},
    urlcolor={blue!60!black}
}

\usepackage{graphicx}

\usepackage[open,openlevel=2,atend]{bookmark}
\usepackage{setspace}

\usepackage[abbrev,msc-links,backrefs]{amsrefs} 
\usepackage{doi}

\renewcommand{\PrintDOI}[1]{\doi{#1}}

\usepackage[T1]{fontenc}
\usepackage{lmodern}
\usepackage[babel]{microtype}
\usepackage[english]{babel}

\usepackage{geometry}
\numberwithin{equation}{section}
\numberwithin{figure}{section}

\usepackage{enumitem}
\def\rmlabel{\upshape({\itshape \roman*\,})}

\def\alabel{\upshape({\itshape \alph*\,})}

\def\nlabel{\upshape({\itshape \arabic*\,})}
\usepackage{tikz}
\usetikzlibrary{calc,decorations.pathmorphing}
\pgfdeclarelayer{background}
\pgfdeclarelayer{foreground}
\pgfdeclarelayer{front}
\pgfsetlayers{background,main,foreground,front}

\let\polishlcross=\l
\def\l{\ifmmode\ell\else\polishlcross\fi}

\def\paragraph#1{%
  \noindent\textbf{#1.}\enspace}

\let\setminus=\smallsetminus

\makeatletter
\def\moverlay{\mathpalette\mov@rlay}
\def\mov@rlay#1#2{\leavevmode\vtop{   \baselineskip\z@skip \lineskiplimit-\maxdimen
   \ialign{\hfil$\m@th#1##$\hfil\cr#2\crcr}}}
\newcommand{\charfusion}[3][\mathord]{
    #1{\ifx#1\mathop\vphantom{#2}\fi
        \mathpalette\mov@rlay{#2\cr#3}
      }
    \ifx#1\mathop\expandafter\displaylimits\fi}
\makeatother

\DeclareFontFamily{U}  {MnSymbolC}{}
\DeclareSymbolFont{MnSyC}         {U}  {MnSymbolC}{m}{n}
\DeclareFontShape{U}{MnSymbolC}{m}{n}{
    <-6>  MnSymbolC5
   <6-7>  MnSymbolC6
   <7-8>  MnSymbolC7
   <8-9>  MnSymbolC8
   <9-10> MnSymbolC9
  <10-12> MnSymbolC10
  <12->   MnSymbolC12}{}
\DeclareMathSymbol{\powerset}{\mathord}{MnSyC}{180}

\let\epsilon=\varepsilon

\let\rho=\varrho
\let\theta=\vartheta

\theoremstyle{plain}
\newtheorem{thm}{Theorem}[section]

\newtheorem{clm}[thm]{Claim}
\newtheorem{fact}[thm]{Fact}

\newtheorem{lem}[thm]{Lemma}

\theoremstyle{definition}

\newtheorem{dfn}[thm]{Definition}

\newtheorem{conj}[thm]{Conjecture}

\usepackage{accents}
\newcommand{\seq}[1]{\accentset{\rightharpoonup}{#1}}
\let\lra=\longrightarrow
\let\ra=\rightarrow

\let\phi=\varphi

\DeclareSymbolFont{stmry}{U}{stmry}{m}{n}
\DeclareMathSymbol\arrownot\mathrel{stmry}{"58}
\DeclareMathSymbol\Arrownot\mathrel{stmry}{"59}
\def\longarrownot{\mathrel{\mkern5.5mu\arrownot\mkern-5.5mu}}

\def\nlra{\longarrownot\longrightarrow}

\let\llra=\longleftrightarrow
\let\vn=\varnothing

\def\od{d^{+}}

\begin{document}

\title[Two disjoint cycles in digraphs]{Two disjoint cycles in digraphs}

\author[Miko\l aj Lewandowski]{Miko\l aj Lewandowski}
\address{Faculty of Computing and Telecommunications, Pozna\'n University of Technology, Pozna\'n, Poland}
\email{mikolaj.lewandowski@student.put.poznan.pl}

\author[Joanna Polcyn]{Joanna Polcyn}
\address{Faculty of Mathematics and Computer Science, Adam Mickiewicz University, Pozna\'n, Poland}
\email{joaska@amu.edu.pl}

\author[Christian Reiher]{Christian Reiher}
\address{Fachbereich Mathematik, Universit\"at Hamburg, Hamburg, Germany}
\email{Christian.Reiher@uni-hamburg.de}


\begin{abstract} 
	Bermond and Thomassen conjectured that every digraph with minimum outdegree 
	at least $2k-1$ contains $k$ vertex disjoint cycles. So far the conjecture 
	was verified for $k\le 3$. Here we generalise the question asking for all 
	outdegree sequences which force $k$ vertex disjoint cycles and give the full 
	answer for $k\le 2$.
\end{abstract}

\maketitle

\setcounter{footnote}{1}

\section{Introduction}

In 1963, Corr\'adi and Hajnal~\cite{CH} proved that every undirected graph with at 
least $3k$ vertices and minimum degree at least $2k$ contains $k$ vertex disjoint 
cycles. In 1981, Bermond and Thomassen~\cite{BT} proposed an analogous conjecture for 
digraphs.

\begin{conj}\label{con}
	For every positive integer $k$ every digraph with minimum outdegree at least $2k-1$ 
	contains $k$ vertex disjoint cycles.
\end{conj}

They also noted that the complete digraph on $2k-1$ vertices shows that the bound 
offered by Conjecture~\ref{con} is optimal. For $k=1$ the problem is easy and the 
case $k=2$ was solved in 1983 by Thomassen~\cite{T}. More than two decades later 
Lichiardopol, P\'or, and Sereni~\cite{LPS} managed to solve the case $k=3$ 
and for all $k>3$ the problem is wide open. It is known, however, that 
Conjecture~\ref{con} holds for tournaments~\cite{BLS1, BLS2}.

The existence of some finite integer $f(k)$ such that every digraph of minimum outdegree 
at least $f(k)$ contains $k$ vertex disjoint cycles was established by Thomassen~\cite{T}. 
Later Alon~\cite{N} proved that it suffices to take $f(k)=64k$.
 
In this paper we generalise the question and ask for \emph{all} degree 
sequences which force the existence of $k$ vertex disjoint cycles. For instance, 
our main result shows that degree sequences 
of the form ${(1, 3, 3, 3, 4, 4, ...)}$ do not force $2$ disjoint cycles, 
but sequences of the form $(1, 3, 3, 3, 5, \dots)$ do. 

\begin{dfn} 
	For a nonempty and nondecreasing sequence $(d_1, \ldots, d_n)$ of nonnegative integers 
	and a positive integer $k$ the relation
	\[
		(d_1, \ldots, d_n)\lra k
	\]
	means that every directed graph on $n$ vertices with outdegree 
	sequence $(d_1, \ldots, d_n)$ 
	contains~$k$ vertex disjoint cycles. The failure of this statement is 
	indicated by
	\[
		(d_1, \ldots, d_n)\nlra k\,.
	\]
\end{dfn}

E.g., Conjecture \ref{con} asserts that $d_1\ge 2k-1$ implies~$(d_1, \dots, d_n)\lra k$. 
In analogy with Chvatal's well-known graph Hamiltonicity theorem one would hope that 
given $k$ there is a somewhat satisfactory characterisation of all outdegree sequences 
forcing $k$ disjoint cycles. Since minimum degree conditions are sometimes difficult 
to maintain in inductive proofs, one may even speculate whether such a strengthened form 
of Conjecture~\ref{con} could be easier to solve than the original question.     
For $k=1$ such a characterisation is easily obtained (see Lemma~\ref{lem:easy}). 
The case $k=2$ requires the following concept. 
 
\begin{dfn}
	Let integers $1\le r\le s\le n$ be given. A sequence  $(d_1,  \ldots, d_n)$ satisfying
	\begin{enumerate}[label=\alabel]
		\item\label{it:a} $d_r\ge r$, $d_s\ge s+1$, and 
		\item\label{it:b} if $n\ge 2s-r+2$ and $d_{2s-r+2}=s+1$, then there 
		is an integer $j\in [2s-r+3, n]$ such that $d_j\ge j$
	\end{enumerate} 
	is called \emph{$(r,s)$-large}. We say that  $(d_1, \ldots, d_n)$ is \emph{large} if it is $(r,s)$-large for some  two integers $r\le s$ in $[n]$.
\end{dfn}

Our main theorem says that being large and forcing two vertex disjoint cycles are equivalent
properties of outdegree sequences. 

\begin{thm}\label{thm:main}
	Let $(d_1,  \ldots, d_n)$ be a nonempty, nondecreasing sequence of nonnegative integers. 
	The relation
	\[
	(d_1,\ldots, d_n)\lra 2
	\]
	holds if and only if the sequence $(d_1, \ldots, d_n)$ is large.
\end{thm}

Note that every nondecreasing outdegree sequence $(d_1,\dots,d_n)$ with $d_1\ge 3=2\cdot 2-1$ 
is $(1,1)$-large. So the case $k=2$ of Conjecture~\ref{con} is a straightforward 
consequence of Theorem~\ref{thm:main}. It would be very interesting to find a similar 
result for three cycles. 

\subsection*{Notation}
A \emph{directed graph} $D$, also called a \emph{digraph}, is a pair $(V(D), E(D))$, 
where $V(D)$ is a set of vertices and $E(D)$ is a set of ordered pairs of vertices 
called \emph{arcs}. Digraphs considered here may have loops and 2-cycles, 
but no parallel arcs.
If an arc $x\ra y$ is present, we say that $x$ \emph{dominates} $y$, 
and $x$ \emph{dominates} a set $W\subseteq V(D)$ if it dominates each vertex of $W$. 
The \emph{inneighbourhood} of $x\in V(D)$ is the set of all vertices dominating $x$ 
and the number such vertices is called the \emph{indegree} of $x$. 
Similarly, the \emph{outneighbourhood} and \emph{outdegree} of $x$ are defined to be, 
respectively, the set of all vertices dominated by $x$ and the number of them.

If $U\subseteq V(D)$ we write $D[U]$ for the subdigraph of $D$ induced by $U$ 
and $D-U$ for the digraph obtained by deleting $U$ (and all arcs starting or 
ending in $U$).

By a cycle we always mean a directed cycle, that is an oriented path starting and ending at 
the same vertex. A cycle of length $\ell$ is an \emph{$\ell$-cycle}. 
A 1-cycle is a \emph{loop}.  In what follows, the outdegree sequence of every digraph is 
nondecreasing. 

A \emph{transitive tournament $T_n$} is a digraph
whose vertex set can be enumerated such that $V(T_n)=\{v_1, v_2, \dots, v_n\}$ and
$
E(T_n) = \{v_i\ra v_j\colon  {v_i, v_j\in V(T_n)}  \textrm{ and }i>j\}.
$

\section{Preliminary results}

We begin by describing all outdegree sequences of vertices of a digraph that force the existence of a cycle. 

\begin{lem}\label{lem:easy}
	The statement $(d_1, \ldots, d_n)\lra 1$ is true if and only if for 
	some $j\in [n]$ the inequality $d_j\ge j$ holds.
\end{lem}

\begin{proof}	
The transitive tournament $T_n$ exemplifies $(0, 1, \dots, n-1)\nlra 1$ and, 
therefore, the relation $(d_1, \ldots, d_n) \lra 1$ entails $d_j \ge j$ for some $j\in [n]$.
	
If, conversely, $d_j\ge j$ holds for some $j\in [n]$, we delete the vertices with outdegrees $d_1, \ldots, d_{j-1}$. The minimum outdegree of the remaining digraph is at least $1$ and,
hence, it contains a cylce.
\end{proof}

Next we show that deleting one term from a sequence cannot destroy its largeness. 

\begin{fact}\label{f:good}
	Let integers $1\le r\le s<n$ be given. If a nondecreasing sequence ${\seq d= (d_1,\ldots,d_n)}$ with $d_n<n$ is $(r,s)$-large, then every sequence $\seq e= (e_1, \ldots, e_{n-1})$ obtained from $\seq d$ by deleting one arbitrary term is also $(r,s)$-large.
\end{fact}
\begin{proof}
	Since  $s\le n-1$ and $e_i\ge d_i$ for all $i\in [n-1]$, the sequence $\seq e$ satisfies  \ref{it:a}. 
	To show that it also satisfies \ref{it:b}, assume $n-1 \ge 2s -r+2$ and $e_{2s-r+2} = s+1$. This implies 
	$
	s+1\le  d_s\le d_{2s-r+2} \le e_{2s-r+2} = s+1,
	$
	and therefore, in view of \ref{it:b} applied to $\seq d$, there is an 
	integer $j\in [2s-r+3, n]$ such that $d_j\ge j$. Because of $d_n< n$ we infer $j\le n-1$, which yields $e_j\ge d_j\ge j$, as required. 
\end{proof}

\section{Proof of Theorem \ref{thm:main}} 

	Our goal in this section is to establish Theorem \ref{thm:main}. 
	Let us first assume ${(d_1, \ldots, d_n)\lra 2}$. 
	The digraph depicted in Figure~\ref{fig:1} shows $(1,\ldots,n)\not\lra 2$. 
	
		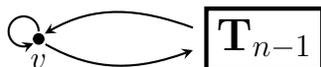
\begin{figure}[h!]
		\begin{tikzpicture}[scale = 1]		
			\filldraw[black] (0,0) circle (2pt) {};
			\node (T) at (3,0) [draw,ultra thick] {\Large $\mathbf T_{n-1}$};
			\draw[thick, -] (0,.1) arc (0:301:.2);
			\draw[thick, -stealth] (-.1,-.07) -- (-0.07,-.058);
			\draw[thick, -stealth] (0.09,-.07) to[out=-30,in=200]   (2.05,-.15);
			\draw[thick, stealth-] (0.08,.07) to[out=30,in=160]   (2.05,.15);
			\node (v) at (0,-.3) {$v$};			
		\end{tikzpicture}
		\caption{A transitive tournament $T_{n-1}$ joined by  2-cycles to a single vertex $v$ with a loop. An arc from/to a box goes from/to every vertex of the box. Each directed cycle in this digraph goes through $v$.}
		\label{fig:1}
	\end{figure}
	
	Thus there is a smallest number $s\in [n]$ such that $d_s\ge s+1$. Now the set
	\[
		X=\{i\in [s]\colon d_i\ge i\}
	\]
	cannot be empty and there exists $r=\min(X)$. 	
	The numbers $r$ and $s$ obey condition~\ref{it:a}. In case $n\ge 2s-r+2$ the graph 
	shown in Figure \ref{fig:2} exemplifies
	\[
		(\textcolor{green!50!black}{0, 1, \ldots, r-2,}\textcolor{blue!60!black}{ r, r+1, \ldots, s-1,}
		\textcolor{red!50!black}{\underbrace{s+1, \ldots, s+1}_{s-r+3},} 2s-r+2, \ldots, n-1)\not\lra 2
	\]
 and thus they cannot violate~\ref{it:b}.
	
	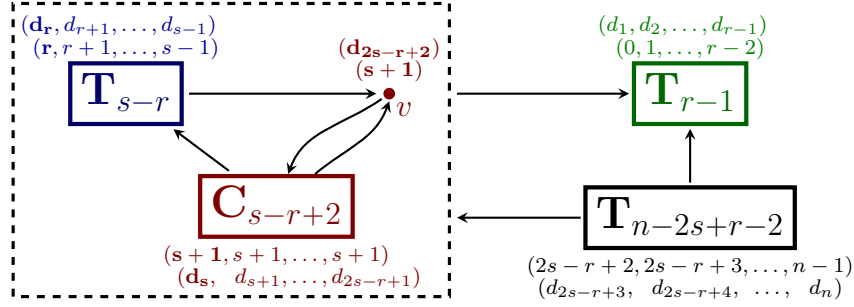
\begin{figure}[h!]
	\begin{tikzpicture}[scale = 1]		
		\draw[black, dashed, very thick] (-3.5,-1.2) rectangle (2.3,2.7);
		\node (r1) at (5.5,1.5) [green!40!black,draw,ultra thick] {\Large $\mathbf T_{r-1}$};
		\filldraw[red!50!black] (1.5,1.5) circle (2pt) {};
		\node (sr) at (-2,1.5) [blue!40!black,draw,ultra thick] {\Large $\mathbf T_{s-r}$};
		\node (sr2) at (0,0) [red!50!black, draw,ultra thick] {\Large$\mathbf C_{s-r+2}$};
		\node (nsr) at (5.5,-.15) [draw,ultra thick] {\Large $\mathbf T_{n-2s+r-2}$};
		\draw[thick, shorten <=2pt,shorten >=2pt,-stealth] (sr2) -- (sr);
		\draw[thick, shorten <=2pt,shorten >=2pt,-stealth] (sr) -- (1.4,1.5);
	
		\draw[thick, -stealth] (1.4,1.43)to[out=215,in=70]   (sr2);
		\draw[thick, stealth-](1.5,1.4) to[out=250,in=40]   (sr2);	
		
		\draw[thick,shorten <=1pt, -stealth] (nsr) -- (2.4,-.15);
		\draw[thick,shorten <=1pt,shorten >=1pt, -stealth] (nsr) -- (r1);
		\draw[thick, shorten >=1pt,-stealth] (2.4,1.5) -- (r1);
		
		\node at (5.5, 2.1) {\tiny $\textcolor{green!50!black}{(0,1,\dots, r-2)}$};
		\node at (5.4, 2.4) {\tiny $\textcolor{green!50!black}{(d_1,d_2,\dots, d_{r-1})}$};
		\node at (1.5, 1.8) {\tiny $\textcolor{red!50!black}{(\mathbf{s+1})}$};
		\node at (1.5, 2.1) {\tiny $\textcolor{red!50!black}{(\mathbf{d_{2s-r+2}})}$};
		\node at (0,-.65) {\tiny $\textcolor{red!50!black}{(\mathbf{s+1}, s+1, \dots, s+1)}$};
		\node at (0.3,-.95) {\tiny $\textcolor{red!50!black}{(\mathbf{d_s},\  \ d_{s+1}, \dots, d_{2s-r+1})}$};
		\node at (-2,2.1) {\tiny $\textcolor{blue!60!black}{(\mathbf{r},r+1, \dots, s-1)}$};
		\node at (-2.1,2.4) {\tiny $\textcolor{blue!60!black}{(\mathbf{d_r},d_{r+1}, \dots, d_{s-1})}$};
		\node at (5.5,-0.8) {\tiny $(2s-r+2, 2s-r+3, \dots, n-1)$};
		\node at (5.5,-1.1) {\tiny $(d_{2s-r+3}, \ \ d_{2s-r+4},\ \,\dots, \ \ d_n)$};
		\node at (1.7, 1.3) [red!50!black]{$v$};		
	\end{tikzpicture}
	\caption{A digraph $D$ obtained from three transitive tournaments, $T_{r-1}$, $T_{s-r}$, $T_{n-2s+r-2}$, one directed cycle $C_{s-r+2}$ and one single vertex $v$. An arc from/to a box goes from/to every vertex of the box. Each directed cycle different from $C_{s-r+2}$ in $D$ goes through $V(C_{s-r+2})$ and $v$. The outdegree sequence of vertices of $D$ is described.}
	\label{fig:2}
\end{figure}

	To verify the converse direction, we consider a counterexample, given by a directed
	graph $D=(V, E)$ on $n$ vertices whose outdegree sequence $(d_1, \ldots, d_n)$ 
	is $(r,s)$-large for two integers $r\le s$ in $[n]$,
	but which fails to contain two disjoint cycles. We may assume that the triple $(D, r, s)$ 		has been chosen in such a way that $|V|+|E|+s$ is as small as possible. 
	The desired contradiction will emerge after nine preliminary claims.

	\begin{clm} 
		We have $r=1$.
	\end{clm} 

	\begin{proof}
		Otherwise delete the vertices with degrees $d_1, \ldots, d_{r-1}$ from $D$, thus 
	obtaining a smaller digraph $D^\star$. Observe that for every $r \le j \le n$ there are at least $n-(j-1)$ vertices in $D^\star$ with outdegree at least $d_{j}-(r-1)$. Therefore, the nondecreasing outdegree sequence $(e_1, \ldots, e_{n-r+1})$ of $D^\star$ has 
	the property $e_k\ge d_{k+r-1}-(r-1)$ holds for every $k\in [n-r+1]$.
	Consequently, $(D^\star, 1, s-r+1)$ is another counterexample to our claim that 
	contradicts the minimality of $(D, r, s)$. 
	\end{proof}

	Notice that now our conditions~\ref{it:a} and~\ref{it:b} read
	\begin{enumerate}[label=\rmlabel]
		\item\label{it:1} $d_1\geq 1$, $d_s\ge s+1$, and 
		\item\label{it:2} if $n\ge 2s+1$ and $d_{2s+1}=s+1$, then there is an 
			integer $j\in [2s+2, n]$ such that $d_j\ge j$.
	\end{enumerate} 
	Exploiting the minimality of $|E|$ we infer
	\begin{enumerate}[label=\rmlabel, resume]
		\item\label{it:3} If $i\in [s-1]$, then $d_i=1$.
		\item\label{it:4} If $i\in [s, \min(2s, n)]$, then $d_i=s+1$.
	\end{enumerate}

	\begin{clm}\label{clm:loop}
		There are no loops in $D$. In particular, $d_n < n$. 
	\end{clm}

	\begin{proof}
		Assume that $D$ possesses a loop at some vertex $x$. Then by~\ref{it:1} 
		the directed graph $D^\star=D- x$ has at most $s-1$ vertices whose outdegree is 
		smaller than $s$ and, moreover, as $n\ge d_n\ge s+1$, it has a further vertex whose outdegree is at 
		least $s$. So by Lemma~\ref{lem:easy}~$D^\star$ contains a cycle, 
		which together with the loop at $x$ yields two disjoint cycles in $D$.
	\end{proof}

Combining the above claim with \ref{it:1} one gets $s+1\le d_s\le d_n< n$, and thereby
\begin{equation}\label{eq:sn}
	n \ge s+2.
\end{equation}

	\begin{clm}\label{clm:dom} 
		Every $2$-cycle of $D$ is dominated by a vertex of outdegree $s+1$.
	\end{clm}

	\begin{proof}
		Note that \eqref{eq:sn} together with \ref{it:1} tell us that $D$ has at least three vertices whose outdegrees are at least $s+1$. If $x\llra y$ is 
		a $2$-cycle in $D$ not dominated by any vertex of outdegree $s+1$, 
		then for the same reason as before the digraph $D-\{x, y\}$ has to contain a cycle. 
	\end{proof}

	\begin{clm}\label{clm:contract}
		Suppose that an arc $x\ra y$ of $D$ does not appear in a $2$-cycle. 
		\begin{enumerate}[label=\nlabel]
			\item\label{it:con1} There is some vertex $a\not\in\{x, y\}$ dominating $x$ and $y$.
			\item\label{it:con2} If the outdegree of $x$ is $1$, then at least $s+1$ 
				vertices distinct from $y$ and having outdegree $s+1$ dominate $\{x, y\}$.
		\end{enumerate}
	\end{clm}

	\begin{proof}
		We construct a digraph $D^\star$ from $D-x$ by adding all arcs of the form $z\ra y$, $z\in V(D^{\star})\setminus \{y\}$, whenever $z\ra y\notin E(D)$ and $z\ra x \in E(D)$ (see Figure \ref{fig:Gstar}).
	Plainly~$D^\star$ cannot contain two disjoint cycles. Observe that the only vertices whose outdegree is smaller in $D^\star$ than in $D$ are those dominating $\{x,y\}$. Their outdegree dropped by one. 
	
	\begin{figure}[h!]
	\begin{tikzpicture}[scale = .4]	
		\coordinate (x) at (-6,2.5);
		\coordinate (y)	at (-3,2.5);
		\coordinate (z) at (-4.5,0);		
		\foreach \i in {x,y,z}
		\fill (\i) circle (3pt);						
		\node at ($(x)+(-.3,.3)$) {  $x$};
		\node at ($(y)+(.4,.3)$) { $y$};
		\node at ($(z)+(.3,-.3)$) { $z$};		
		\node at (-7,.5) {\Large $D$};
		\node at (2,.5) {\Large $D^{\star}$};		
		\draw [ thick,shorten <=2pt,shorten >=2.5pt,-stealth] (x) to[out=20,in=160] (y);
		\draw [ thick,shorten <=2pt,shorten >=2.5pt,-stealth] (z) to[out=140,in=280] (x);		
		\coordinate (x1) at (3,2.5);
		\coordinate (y1)	at (6,2.5);
		\coordinate (z1) at (4.5,0);		
		\foreach \i in {y1,z1}
		\fill (\i) circle (3pt);		
		\fill [black!20!white](x1) circle (3pt);				
		\node [black!20!white] at ($(x1)+(-.3,.3)$) {  $x$};
		\node at ($(y1)+(.4,.3)$) {$y$};
		\node at ($(z1)+(.3,-.3)$) { $z$};		
		\node at (0,1.2) {\LARGE $\Rightarrow$};	
		\draw [thick,shorten <=2pt,shorten >=2.5pt,-stealth] (z1) to[out=40,in=260] (y1);	
	\end{tikzpicture}
	\caption{New arcs in $D^{\star}$.}
	\label{fig:Gstar}
\end{figure}

	If~\ref{it:con1} fails, then the outdegree sequence of $D^\star$ is the same as that 
	of $D$ with the term corresponding to $x$ removed and thus, in view of Fact \ref{f:good} and Claim \ref{clm:loop}, it is $(r,s)$-large, meaning that $(D^\star, r, s)$ contradicts the supposed minimality of $(D, r,s)$. 
	
	Let us now suppose that $x$ has outdegree $1$. Due to~\ref{it:4} this is only possible 
	if $s>1$. By~\ref{it:3} the digraph $D^\star$ has $s-2$ vertices whose outdegree is $1$, 
	while~\ref{it:4} shows that each other vertex has the outdegree at least $s$. 
	Now the only possibility for the triple $(D^\star, 1, s-1)$ not to contradict the 
	supposed minimality of $(D, 1, s)$ is that clause~\ref{it:2} fails, which means, 
	in particular, that $n-1\ge 2(s-1)+1$, i.e., $n\geq 2s$, and that $d_{2s-1}=s$ meaning $D^\star$ has $(2s-1)-(s-2)=s+1$ 
	vertices of outdegree $s$. In $D$ these vertices need to have outdegree $s+1$ and 
	thus they dominate $\{x, y\}$. This proves~\ref{it:con2}. 
	\end{proof}

	\begin{clm}\label{clm:incycle}
		The inneighbourhood of every vertex of $D$ contains a cycle.
	\end{clm}

	\begin{proof}
		Let $y$ denote any vertex of $D$ and let $N$ be its inneighbourhood. 

		Assume first that $N=\varnothing$. Then the outdegree sequence of $D^\star=D- y$ 
		is obtained from that of $D$ by removing the term corresponding to $y$ and, therefore, in view of Fact \ref{f:good} and Claim \ref{clm:loop}, is $(1,s)$-large. Thus, the 
		triple $(D^\star, 1, s)$ contradicts the minimality of $(D, 1, s)$.

		Thereby we have shown $N\ne\vn$. By Lemma~\ref{lem:easy} our claim can only fail 
		if some vertex $x$ of $D[N]$ has no inneighbours in that graph. By Claim~\ref{clm:dom} 
		this is only possible if there is no arc from $y$ to $x$, but this 
		contradicts Claim~\ref{clm:contract}\ref{it:con1}.
	\end{proof}

	\begin{clm}\label{clm:2ex}
		There is a $2$-cycle in $D$.
	\end{clm}

	\begin{proof}
		Otherwise, in view of Claim~\ref{clm:incycle}, each vertex of $D$ has indegree at 
		least $3$. Let $D^\star=(V, E^\star)$ be a directed graph arising from $D$ by 
		first reversing 
		all arrows and then deleting an arbitrary arc. Its outdegree sequence 
		$(e_1, \ldots, e_n)$ has the properties $e_1\ge 2$ and $e_3\ge e_2\ge 3$ and thus it is $(1,1)$-large.
		Since $|V|+|E^\star|+1 < |V|+|E|+s$, the minimality of $(D, r, s)$ tells us 
		that $D^\star$ and thus also $D$ contains two disjoint cycles, which is 
		a contradiction. 
	\end{proof}

	\begin{clm}
		In $D$ there are a directed cycle $C$ all of whose vertices have outdegree $s+1$ 
		and a vertex $x\not\in V(C)$ connected to every vertex of $C$ by a $2$-cycle.
	\end{clm}  

	\begin{proof} 
		Let us first assume that $D$ has a vertex $x$ belonging to all $2$-cycles. Define
		\begin{align*}
			A&=\{z\in V\colon z\ne x \text{ and } x\llra z\text{ is a $2$-cycle}\} \\
			\text{ and }
			B&=\{z\in A\colon \text{the outdegree of $z$ is $s+1$}\}.
		\end{align*}
		Claim~\ref{clm:2ex} informs us that $A\ne\vn$. 
		Consider any $a\in A$. By Claim~\ref{clm:dom} the $2$-cycle $a\llra x$ is dominated 
		by some vertex $b\not\in\{a,x\}$ whose outdegree is $s+1$. The inneighbourhood of~$b$ 
		contains some cycle by Claim~\ref{clm:incycle}. This cycle cannot be disjoint 
		to $a\llra x$ and thus there is an arc from $a$ or $x$ to $b$. The former is 
		impossible, for then the $2$-cycle $a\llra b$ would not pass through~$x$ and for 
		this reason we must have $b\in B$ (see Figure \ref{fig:axb}). 

	\begin{figure}[h!]
	\begin{tikzpicture}[scale = .6]
		\coordinate (x) at (-6,-2.5);
		\coordinate (a)	at (-3,-2.5);
		\coordinate (b) at (-4.5,0);		
		\foreach \i in {x,a,b}
		\fill (\i) circle (3pt);						
		\node at ($(x)+(-.35,-.1)$) { $x$};
		\node at ($(a)+(1.45,-.1)$) {$a\quad a \in A$};
		\node at ($(b)+(2.7,.25)$) { $b\quad \od_D(b)=s+1$};		
		\draw [ thick,shorten <=3pt,shorten >=3pt,-stealth] (x) to[out=40,in=150] (a);
		\draw [ thick,shorten <=3pt,shorten >=3pt,-stealth] (b) to[out=200,in=100] (x);
		\draw [ thick,shorten <=3pt,shorten >=3pt,-stealth] (b) to[out=-10,in=80] (a);		
		\draw [  thick,shorten <=3pt,shorten >=3pt,-stealth] (a) to[out=210,in=-30] (x);
		\draw [ shorten <=3pt,shorten >=3pt,-stealth] (x) to[out=60,in=270] (b);
	\end{tikzpicture}
	\caption{A 2-cycle $a\llra x$ and a vertex $b$ dominating $\{a,x\}$.}
	\label{fig:axb}
\end{figure}

		We have thus shown that every member of $A$ has an inneighbour in $B\subseteq A$. 
		In particular, $B\ne\vn$ and in the restriction $D[B]$ every vertex has indegree at 
		least $1$. So, in view of Lemma \ref{lem:easy}, there is some cycle $C$ in $B$ and the claim follows in this
		case.

		From now on we may assume that $D$ does not have a vertex belonging to all of 
		its $2$-cycles. This means that the undirected graph $G=(V,E_G)$ the 
		edges of which correspond to the $2$-cycles of $D$ is not the disjoint union of 
		a (possibly degenerate) star and isolated vertices. As $G$ cannot have two 
		independent edges either, it has to be the disjoint union of a triangle and isolated 
		vertices. 
		In other words, there are three distinct vertices $x$, $y$, and~$z$ of~$D$ such 
		that $x\llra y$, $x\llra z$, and $y\llra z$ are $2$-cycles. 

		If both $y$ and $z$ have outdegree $s+1$ the claim still holds with $y\llra z$ 
		playing the r\^{o}le of $C$. So let us finally assume that the outdegree of $y$ 
		is not $s+1$. Applying Claim~\ref{clm:dom} to $x\llra z$ we get a vertex 
		$a\not\in\{x, y, z\}$ dominating $x$ and $z$. Since the cycle contained in 
		the inneighbourhood of $a$ has to intersect $\{x,z\}$ we may further suppose 
		that $x$ dominates $a$. But now $a\llra x$ and $y\llra z$ are two disjoint $2$-cycles 
		in $D$ (see Figure \ref{fig:triangle}).  
		
		\begin{figure}[h!]
		\begin{tikzpicture}[scale = .6]	
			\coordinate (x) at (-6,2.5);
			\coordinate (y)	at (-3,2.5);
			\coordinate (z) at (-4.5,0);
			\coordinate (a) at (-7,.3);			
			\foreach \i in {x,y,z,a}
			\fill (\i) circle (3pt);						
			\node at ($(x)+(-.3,.35)$) {  $x$};
			\node at ($(y)+(.4,.25)$) {  $y$};
			\node at ($(z)+(.3,-.3)$) {  $z$};
			\node at ($(a)+(-.4,.2)$) {  $a$};			
			\draw [ thick,shorten <=3pt,shorten >=4pt,-stealth] (x) to[out=40,in=150] (y);
			\draw [  thick,shorten <=3pt,shorten >=4pt,-stealth] (z) to[out=160,in=260] (x);
			\draw [ thick,shorten <=3pt,shorten >=4pt,-stealth] (y) to[out=280,in=20] (z);
			
			\draw [ thick,shorten <=3pt,shorten >=4pt,-stealth] (y) to[out=190,in=-10] (x);
			\draw [  thick,shorten <=3pt,shorten >=4pt,-stealth] (x) to[out=310,in=110] (z);
			\draw [ thick,shorten <=3pt,shorten >=4pt,-stealth] (z) to[out=70,in=230] (y);
			\draw [ thick,shorten <=3pt,shorten >=4pt,-stealth] (a) to[out=90,in=190] (x);
			\draw [ thick,shorten <=3pt,shorten >=4pt,-stealth] (a) to[out=-30,in=200] (z);
			\draw [ shorten <=3pt,shorten >=4pt,-stealth] (x) to[out=220,in=20] (a);	
		\end{tikzpicture}
		\caption{2-cycles $x\llra y$, $y\llra z$, $y\llra x$ in $D$ and a vertex $a$ dominating $\{x,z\}$.}
		\label{fig:triangle}
	\end{figure}
	
	\end{proof}

	From now on, we fix $C$ and $x$ as obtained by the previous claim. Observe that the 
	cycle~$C$ has to be induced, for otherwise $D$ would contain two disjoint cycles.

	\begin{clm}
		We have $s>1$.
	\end{clm}

	\begin{proof}
		Assume $s=1$. If $D$ has no vertices besides $x$ and those on $C$, 
		then the outdegree sequence of $D$ has to be
		\[
			(\underbrace{2, \ldots, 2}_{n-1}, n-1)
		\]
		for some $n\ge 3$, contrary to~\ref{it:2}. So there is a vertex $y\ne x$ not lying 
		on $C$. By Claim~\ref{clm:incycle} the inneighbourhood of $y$ contains a cycle. 
		 This cycle needs to have a common vertex $z$ with $C$. But now $x$, $y$ and the 
		successor of $z$ on $C$ are three distinct outneighbours of $z$, contrary to the 
		fact that the members of $C$ have outdegree $s+1=2$ (see Figure \ref{fig:Cx}).   
		
		\begin{figure}[h!]
		\begin{tikzpicture}[scale = .9]	
			\foreach \i in {0,...,6} {
				\pgfmathsetmacro\s{72*\i+18};
				\coordinate (v\i) at (\s:1.5cm);
				\fill (v\i) circle (1.6pt);
			}	
			\foreach \i/\j in {0/1, 1/2, 2/3, 3/4, 4/5, 5/6}
			\draw [green!60!black ,thick, shorten <=3pt,shorten >=3pt,stealth- ](v\i) -- (v\j);
			\coordinate (x) at (0,0);
			\fill (x) [blue]circle (1.8pt);	
			\draw [shorten <=3pt, shorten >= 3pt, stealth-] (x) [out = 108, in = -105] to (v1);
			\draw [shorten <=3pt, shorten >= 3pt, -stealth] (x) [out = 72, in = -75] to (v1);		
			\draw [shorten <=3pt, shorten >= 3pt, stealth-] (x) [out = 180, in = -35] to (v2);
			\draw [shorten <=3pt, shorten >= 3pt, -stealth] (x) [out = 144, in = -5] to (v2);		
			\draw [shorten <=3pt, shorten >= 3pt, stealth-] (x) [out = 252, in = 40] to (v3);
			\draw [shorten <=3pt, shorten >= 3pt, -stealth] (x) [out = 216, in = 70] to (v3);		
			\draw [shorten <=3pt, shorten >= 3pt, stealth-] (x) [out = 324, in = 115] to (v4);
			\draw [shorten <=3pt, shorten >= 3pt, -stealth] (x) [out = 288, in = 145] to (v4);		
			\draw [thick, shorten <=3pt, shorten >= 3pt, stealth-] (x) [out = 36, in = 185] to (v0);
			\draw [shorten <=3pt, shorten >= 3pt, -stealth] (x) [out = 0, in = 215] to (v0);		
			\node at (-.3,.4) {$x$};
			\node at (1.55, .26) {$z$};	
			\coordinate (y) at (2.1,1.1);
			\fill  (y) [blue] circle (1.7pt);
			\node  at (2.2, .85) {$y$};
			\draw [ thick, ,shorten <=3pt,shorten >=3pt,stealth- ](y) -- (v0);			
			\fill (v4)[blue] circle (1.7pt);			
			\node [green!60!black] at (-1.6,-.5) {\Large $C$};
		\end{tikzpicture}
		\caption{The cycle \textcolor{green!60!black}{$C$} and the vertices $x$ and $y$.}
		\label{fig:Cx}
	\end{figure}
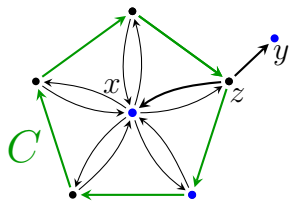
		
	\end{proof}

	\begin{clm}\label{clm:final}
		If $n\ge 2s+1$, then $d_{2s+1}\ge s+2$. In particular, $D$ has at most $s+1$ vertices 
		of outdegree $s+1$.
	\end{clm}

	\begin{proof}
		Assume for the sake of contradiction that $n\geq 2s+1$ and $d_{2s+1} < s+2$. Since $d_{2s+1}\ge d_{s}=s+1$, 
		we have $d_{2s+1}=s+1$, meaning that~\ref{it:2} yields the existence 
		of some $j\in [2s+2, n]$ with $d_j\ge j$. Let $D^\star$ be the digraph that arises 
		from $D$ if its $s-1$ vertices of outdegree $1$ are removed. The nondecreasing outdegree 
		sequence $(e_1, \ldots, e_{n-s+1})$ of $D^\star$
		satisfies $e_1\geq 2$ and $e_{j-s+1}\ge j-s+1$. Since $j-s+1\ge s+3\ge 4$ this tells 
		us that the triple $(D^\star, 1, 1)$ violates the minimality of $(D, 1, s)$.
	\end{proof}

	Notice that the graph $D- V(C)$ cannot contain a cycle. Since it is not empty, it needs 
	to contain some vertex $z$ without any outneighbours by Lemma~\ref{lem:easy}. 

	If $z$ had outdegree $1$ in $D$, then $z\neq x$ and its unique outneighbour $a$ had to lie on $C$. 
	In particular, the outdegree of $a$ had to be $s+1$, which in view of 
	Claim~\ref{clm:contract}\ref{it:con2} entailed that~$D$ had at least $s+2$ vertices 
	the outdegree of which is $s+1$, contrary to Claim~\ref{clm:final}.
		
	By~\ref{it:3},~\ref{it:4}, this tells us 
	that the outdegree $\ell$ of $z$ in $D$ is consequently at least $s+1$. Since $z$ has 
	no outneighbours in $D-V(C)$, the length of~$C$ needs to be at least $\ell$ and possibly 
	together with $z$ itself this yields $s+2$ vertices of $D$ whose outdegree is $s+1$. 
	We have thereby reached the final contradiction that concludes the proof of 
	Theorem~\ref{thm:main}.~\qed


\begin{bibdiv}
\begin{biblist}

\bib{N}{article}{
   author={Alon, Noga},
   title={Disjoint directed cycles},
   journal={Journal of Combinatorial Theory, Series B},
   volume={68},
   date={1996},
   number={2},
   pages={167--178},
   issn={0095-8956},
   review={\MR{1417794}},
   doi={10.1006/jctb.1996.0062},
}
	
	\bib{BT} {article}{
		author = {Bermond, Jean Claude},
		author = {Thomassen, Carsten},
		title={Cycles in digraphs--a survey},
		journal={Journal of Graph Theory},
		volume={5},
		number={1},
		pages={1--43},
		year={1981},
		publisher={Wiley Online Library}
	}

\bib{BLS1}{article}{
	title={Two proofs of Bermond-Thomassen conjecture for regular tournaments},
	author={Bessy, St{\'e}phane}, 
author = { Lichiardopol, Nicolas},
author={Sereni, Jean S{\'e}bastien},
	journal={Electronic Notes in Discrete Mathematics},
	volume={28},
	pages={47--53},
	year={2007},
	publisher={Elsevier}
}

\bib{BLS2}{article}{
	title={Two proofs of the Bermond--Thomassen conjecture for tournaments with bounded minimum in-degree},
	author={Bessy, St{\'e}phane}, 
	author = { Lichiardopol, Nicolas},
	author={Sereni, Jean S{\'e}bastien},
	journal={Discrete Mathematics},
	volume={310},
	number={3},
	pages={557--560},
	year={2010},
	publisher={Elsevier}
}
	
	\bib{CH}{article}{
		title={On the maximal number of independent circuits in a graph},
		author={Corr\'adi, Kereszt{\'e}ly},
		author= {Hajnal, Andr{\'a}s},
		journal={Acta Mathematica Hungarica},
		volume={14},
		number={3-4},
		pages={423--439},
		year={1963},
		publisher={Akad{\'e}miai Kiad{\'o}, co-published with Springer Science+ Business Media BV~É}
	}
	
\bib{Chv}{article}{
   author={Chv\'{a}tal, V.},
   title={On Hamilton's ideals},
   journal={J. Combinatorial Theory Ser. B},
   volume={12},
   date={1972},
   pages={163--168},
   issn={0095-8956},
   review={\MR{294155}},
   doi={10.1016/0095-8956(72)90020-2},
}

\bib{LPS}{article}{
   author={Lichiardopol, Nicolas},
   author={P\'{o}r, Attila},
   author={Sereni, Jean S\'{e}bastien},
   title={A step toward the Bermond-Thomassen conjecture about disjoint
   cycles in digraphs},
   journal={SIAM Journal on Discrete Mathematics, Society for Industrial and Applied Mathematics}
   volume={23},
   date={2009},
   number={2},
   pages={979--992},
   issn={0895-4801},
   review={\MR{2519939}},
   doi={10.1137/080715792},
}

\bib{T}{article}{
	title={Disjoint cycles in digraphs},
	author={Thomassen, Carsten},
	journal={Combinatorica},
	volume={3},
	number={3},
	pages={393--396},
	year={1983},
	publisher={Springer}
}

\end{biblist}
\end{bibdiv}
\end{document}